\documentclass[12pt]{article}
\usepackage{amstext    }
\usepackage{amsthm    }
\usepackage{a4}
\usepackage[mathscr]{eucal}
\usepackage{mathrsfs}
\usepackage{indentfirst}
\usepackage{amsmath,bm}

\usepackage{amsmath}
\usepackage{amssymb}
\usepackage{amscd}

\newtheorem{theorem}{Theorem}[section]
\newtheorem{definition}[theorem]{Definition}
\newtheorem{proposition}[theorem]{Proposition}
\newtheorem{corollary}[theorem]{Corollary}
\newtheorem{lemma}[theorem]{Lemma}
\newtheorem{remark}[theorem]{Remark}

\title{On equidistribution theorem for multi-sequences of holomorphic line bundles}
\title{{\Large \bf On equidistribution theorem for multi-sequences of holomorphic line bundles
\thanks{The second author was supported in part by NSFC Grant No. 11901594, the third author was supported by NSFC Grant No. 12001549 and Guangdong Basic and Applied Basic Research Foundation Grant No. 2019A1515110250.}}}

\author{Manli Liu$^1$\ \ \ Weixiong Mai$^2$ \ \ \ Guokuan Shao$^3$}
\AtEndDocument{\bigskip{\footnotesize
 \textsc{$1.$\ \ Department of Mathematics, South China Agricultural University, 510642, China}\par
  \textit{E-mail address: }\texttt{Liuml@scau.edu.cn}\par
  \textsc{$2.$\ \ Macao Center for Mathematical Sciences, Macau University of Science and Technology, Macau, China}\par
  \textit{E-mail address: }\texttt{wxmai@must.edu.mo}\par
  \medskip
  \textsc{$3.$\ \ School of Mathematics (Zhuhai), Sun Yat-sen University, Guangdong, China}\par
  \textit{E-mail address: }\texttt{shaogk@mail.sysu.edu.cn}
}

}
\begin{document}

\maketitle

\begin{abstract}
Given several sequences of Hermitian holomorphic line bundles $\{(L_{kp}, h_{kp})\}_{p=1}^{\infty}$, we establish the distribution of common zeros of random holomorphic sections of $L_{kp}$ with respect to singular measures. We also study the dimension growth for a sequence of pseudo-effective line bundles.
\end{abstract}

\noindent
{\bf Classification AMS 2020:} 32A60, 32L10, 32U40.


\section{Introduction}
The study of equidistribution of zeros of random holomorphic sections has become intensively active recently, which can be applied to quantum chaotic eigenfunctions \cite{bbl, nv} and shed a light on the Quantum unique ergodicity conjecture \cite{rs, hs}.

Shiffman-Zelditeh \cite{sz} established an equidistribution theorem for high powers of a positive line bundle. More results on equidistribution for singular metrics of line bundles were obtained. For example, Dinh-Ma-Marinescu \cite{dmm} explored the equidistribution of zeros of random holomorphic sections for singular Hermitian metrics with a convergence speed. Coman-Marinescu-Nguy\^{e}n \cite{cmn} studied the equidistribution of common zeros of sections of several big line bundles. Coman-Marinescu-Nguy\^{e}n \cite{cmn19} studied equidistribution for spaces of $L^2$-holomorphic sections vanishing along subvarieties. Dinh-Sibony \cite{ds1} first extended equidistribution with respect to general measures with a good convergence speed. Shao \cite{sh1,sh2} provided large family of singular measures to satisfy equidistribution theorems.  See \cite{cm, bcm, bchm, cmm, hs} for more references.

Recently, Coman-Lu-Ma-Marinescu \cite{clmm} studied equidistrbution for a sequence $(L_p, h_p)$ instead of $(L^{\otimes p},h^{\otimes p})$ of a single line bundle $L$. They imposed a natural convergence, that is, the first Chern curvature currents $c_1(L_p,h_p)$ converge to a (non-integral) K\"{a}hler form $\omega$, which can be regarded as a "prequantization" of $\omega$ in the setting of geometric quantization.

In this paper, we establish an equidistribution theorem for several sequences of line bundles. Now we formulate our setting and state our main result. Let $(X, \omega)$ be a K\"{a}hler manifold of $\dim_{\mathbb{C}}X=n$ with a fixed K\"{a}hler form $\omega$.
Let $\{(L_{kp}, h_{kp})\}_{p=1}^{\infty}$ be $m$ sequences of Hermitian holomorphic line bundles on $X$ with (possibly singular) Hermitian metrics $h_{kp}$, where $1\leq k\leq m\leq n$. We endow the space $\mathscr{C}^{\infty}(X,L_{kp})$ of smooth sections $L_{kp}$ with the inner product
\begin{equation*}
\langle s_1,s_2 \rangle:=\int_X \langle s_1,s_2 \rangle_{h_{kp}} \frac{\omega^n}{n!}, s_1, s_2 \in\mathscr{C}^{\infty}(X,L_{kp}),
\end{equation*}
and we set $\|s\|^2=\langle s,s \rangle$. We denote by $L^2(X,L_{kp})$ the completion of $\mathscr{C}^{\infty}(X,L_{kp})$ with respect to this norm. Denote by $H_{(2)}^0(X,L_{kp})$ the Bergman space of $L^2$ holomorphic sections of $L_{kp}$ and let $B_{kp}:=L^2(X,L_{kp})\rightarrow H_{(2)}^0(X,L_{kp})$ be the orthogonal projection. The integral kernel $B_{kp}(x,x')$ of $B_{kp}$ is called the Bergman kernel. The restriction of the Bergman kernel to the diagonal of $X$ is the Bergman kernel function of $H_{(2)}^0(X,L_{kp})$, which we still denote by $B_{kp}$, i.e., $B_{kp}(x)=B_{kp}(x,x)$.
The first assumption is the following:

{\bf Assumption 1:}\ \ There exist a constant $M_0>1$ and $p_0>0$ such that
\begin{equation*}
\frac{A_{kp}^n}{M_0}\leq B_{kp}(x)\leq M_0A_{kp}^n,
\end{equation*}
for any $ x\in X, p\geq p_0, 1\leq k\leq m$, where $A_{kp}$ are positive numbers, $\lim\limits_{p\rightarrow \infty} A_{kp}=\infty$ for $1\leq k\leq m$ with the same order of infinite.

Denote by $\mathbb{CP}H_{(2)}^0(X,L_{kp})$ the associated projective space of $H_{(2)}^0(X,L_{kp})$. Set  $d_{kp}:=\dim\mathbb{CP}H_{(2)}^0(X,L_{kp})$. By Assumption 1, we have
\begin{equation*}
d_{kp}=\int_{X}B_{kp}(x)\omega^{n}-1\approx A_{kp}^n.
\end{equation*}
There exist $M_1>1$ and $p_0>0$, such that
\begin{equation*}
\frac{A_{kp}^n}{M_1}\leq d_{kp}\leq M_1A_{kp}^n.
\end{equation*}

Consider the multi-projective spaces
\begin{equation}\label{e-7283}
\mathbb{X}_p:=\mathbb{CP}H_{(2)}^0(X,L_{1p})\times\cdots \times \mathbb{CP}H_{(2)}^0(X,L_{mp}),
\end{equation}
equipped with a probability (singular) measure $\sigma_{p}$. The standard measure $\sigma_{p}$ is just the product of the Fubini-study volume forms on all components of $\mathbb{X}_p$.
In our main theorem, $\sigma_p$ is the product of moderate measures, which is a generalization of the standard one (cf. \eqref{e-7281}).
The product space with product measure is $\mathbb{P}^X:=\prod\limits_{p=1}^{\infty} \mathbb{X}_p, \sigma=\prod\limits_{p=1}^{\infty} \sigma_{p}$. For $s_p=(s_{1p},\cdots,s_{mp})\in \mathbb{X}_p$, we define $[s_p=0]:=[s_{1p}=0]\wedge\cdots\wedge[s_{mp}=0]$.
With a minor change of the proof for Bertini type theorem with respect to a singular measure in \cite[Section 2]{sh2}. We can deduce that $[s_p=0]$ is well-defined for a.e. $\{s_p\}$ with respect to $\sigma_p$. In fact, if $\sigma_p$ has no mass on any proper analytic subset of $\mathbb{X}_p$, the above statement still holds true, see the proof of \cite[Lemma 2.2, Proposition 2.3]{sh2}.

Now, we give the second natural assumption on the convergence of Chern curvature currents $c_1(L_{kp},h_{kp})$ associated to $h_{kp}$, which can be thought as a "prequantization" process.

{\bf Assumption 2:}\ \  There exist positive closed $(1,1)$-currents $\omega_k$ ($\omega_k\geq 0$ in the sense of currents) with positive measure (i.e.$\int \omega_k\wedge \omega^{n-1}>0$) such that the norms of currents satisfy
\begin{equation*}
\left\|\frac{1}{A_{kp}}c_1(L_{kp},h_{kp})-\omega_k\right\|\leq C_0 A_{kp}^{-a_k},
\end{equation*}
where $C_0, a_k>0$ are constants. Moreover, $\omega_{j_1}\wedge \cdots \wedge \omega_{j_l}$ are well-defined with positive measures, for any multi-index $(j_1,\cdots,j_l)\subset\{1,2,\cdots,m\}$.

Note that one trivial example is $L_{kp}=L_k^{\otimes p}, h_{kp}=h_{k}^{\otimes p}, \omega_k=c_1(L_k,h_k)$ and the non-continuous set of $h_k$ are in general position (cf. \cite{cmn}).

Now we are in a position to state our main theorem:
\begin{theorem}\label{thm1.1}
Let $\{(L_{kp},h_{kp})\}_{p=1}^{\infty}$ be $m$ sequences of Hermitian holomorphic line bundles on a compact K\"{a}hler manifold $X$ of $\dim_{\mathbb{C}}X=n$. If Assumption 1 and Assumption 2 hold, then for $\sigma$-a.e. $\{s_p\}\in \mathbb{P}^{X}$, we have
\begin{equation*}
\frac{1}{\prod\limits_{k=1}^m A_{kp}}[s_p=0]\rightarrow \omega_1\wedge\cdots\wedge \omega_m
\end{equation*}
in the sense of currents.
\end{theorem}
\begin{theorem}\label{thm1.2}
With the same notations and assumptions in Theorem \ref{thm1.1}, for any $\alpha>0, $ there exist $C_1>0, C_2>0, p_1>0$ and $E_p^{\alpha}$ such that

\noindent (i)\quad $\sigma_p(E_p^{\alpha})\leq C_1(\sum\limits_{k=1}^mA_{kp})^{-\alpha}$ for any $p>p_1$;

\noindent (ii)\quad for $s_p\in X_p\backslash E_{p}^{\alpha}$ and any $(n-m,n-m)$ form $\phi$ of class $\mathscr{C}^2$,
\begin{equation*}
\begin{array}{rl}
&|\langle\frac{1}{\prod\limits_{k=1}^m A_{kp}}[s_p=0]-\omega_1\wedge\cdots\wedge \omega_m, \phi\rangle|\\
\leq & C_2 \left(\sum\limits_{k=1}^m
\frac{\log A_{kp}}{A_{kp}}+\frac{\log(\sum\limits_{k=1}^mA_{kp})}{\sum\limits_{k=1}^mA_{kp}}+\sum\limits_{k=1}^mA_{kp}^{-a_{k}}\right)\|\phi\|_{\mathscr{C}^2}.
\end{array}
\end{equation*}
\end{theorem}

In \cite{clmm}, $L_{kp}=L_p, \omega_k=\omega_1$ and $h_{kp}=h_p$ are smooth metrics. In this case, Assumption 1 is automatically satisfied under Assumption 2. In fact, they gave the asymptotic expansion of the Bergman kernel $B_{kp}(x)=B_p(x)$. Our theorem is a generalization of \cite[Theorem 0.4]{clmm}.
If $L_{kp}=L_k^{\otimes p}, \{L_k\}_{k=1}^m$ are all big line bundles and the non-continuous set of $h_{kp}=h_k^{\otimes p}$ are in general position and of H\"{o}lder continuous with singularities, our theorems recover \cite[Theorem 1.2, Theorem 1.3]{sh2} partially.

In the classical setting of equidistribution theorems, there is a single sequence $\{L_p\}$ of a positive line bundle $L$. Assumption 1 is satisfied due to the uniform estimate of $B_p(x)$ for $H^{0}(X,L^p)$, then we derive the classical result by Shiffman-Zelditch \cite{sz}.

\begin{corollary}\label{cor1.3}
Let $m=1, \{L_{p}=L^{\otimes p}\}$ be a sequence of high powers of a positive line bundle. Take $\omega=c_1(L,h)$, where $h$ is a smooth Hermitian metric, $\sigma_p$ is the Fubini-Study volume on $\mathbb{CP}H^{0}(X,L_p)$. Then, for $\sigma$-a.e. $\{s_p\}\in\mathbb{P}^{X}$, $\frac{1}{p}[s_p=0]\rightarrow \omega$ in the sense of currents.
\end{corollary}

Assumption 1 is indeed strong. We next provide a dimension growth result for sequences of pseudo-effective line bundles, which can shed a light on this assumption. To simplify, let $(L_{1p}, h_{1p})=(L_{p}, h_{p})$, where all $L_p$ are pseudo-effective line bundles, i.e. $c_1(L_p,h_p)\geq 0$ in the sense of currents. $A_p>0, \lim\limits_{p\rightarrow \infty}A_p=+\infty$ and $\frac{1}{A_p}c_1(L_p,h_p)\rightarrow \omega_1$.
Suppose that $h_p$ is continuous on $X\backslash \Sigma$, where $\Sigma$ is a proper analytic subset, $c_1(L_p,h_p)\geq A_p\eta_p \omega$, where $\eta_p:X\rightarrow [0,+\infty)$. For any $x\in X\backslash \Sigma$, there exists a neighborhood $U_x$ of $x$ and a constant $c_x>0$ such that $\eta_p(x)\geq c_x$ on $U_x$ for any $p$ large.

\begin{theorem}\label{thm1.4}
With the above setting, there exists a constant $C_3>0$, we have $\dim H_{(2)}^0(X,L_p)\geq C_3 A_p^n$.
\end{theorem}

In the rest of the paper, we abusively use $C$ to denote positive constants, where $C$ is not necessary to be the same in different places.
The paper is organized as follows. In Section 2 we recall Dinh-Sibony's technique and the notion of moderate measure
with estimates for capacities in multi-projective spaces.
In Section 3 we prove a variant of convergence result of Fubini-Study currents for multi-sequences of holomorphic line bundles.
Section 4 is devoted to proving the main theorem. We conclude Section 5 with a dimension growth estimate for pseudo-effective line bundles.

\section{Preliminaries}
\subsection{Dinh-Sibony's technique}
Let $(X, \omega)$ (resp. $(Y, \omega_{Y})$) be a compact K\"{a}hler manifold of dimension $n$ (resp. $n_{Y}$).
Recall that a $meromorphic$ $transform$ $F: X\rightarrow Y$ is the graph of an analytic subset $\Gamma\subset X\times Y$
of pure dimension $n_{Y}+k$ such that the natural projections $\pi_{1}: X\times Y\rightarrow X$
and $\pi_{2}: X\times Y\rightarrow Y$ restricted to each irreducible component of the analytic subset $\Gamma$ are surjective.
We write $F=\pi_{2}\circ (\pi_{1}|_{\Gamma})^{-1}$.
The dimension of the fiber $F^{-1}(y):=\pi_{1}(\pi_{2}^{-1}|_{\Gamma}(y))$ is equal to $k$
for the point $y\in Y$ generic. This is the codimension of the meromorphic transform $F$.
If $T$ is a current of bidegree $(l,l)$ on $Y$, $n_{Y}+k-n\leq l\leq n_{Y}$, we define
\begin{equation}\label{e-7283-1}
F^{\star}(T):=(\pi_{1})_{\ast}(\pi_{2}^{\ast}(T)\wedge [\Gamma]),
\end{equation}
where $[\Gamma]$ is the current of integration over $\Gamma$.
We introduce the notations of intermediate degrees of $F$,
\begin{equation*}
\begin{split}
& \delta^1(F):=\int_{X}F^{\ast}(\omega_{Y}^{n_{Y}})\wedge\omega^{k}, \\
& \delta^2(F):=\int_{X}F^{\ast}(\omega_{Y}^{n_{Y}-1})\wedge\omega^{k+1}. \\
\end{split}
\end{equation*}

To introduce more notions and notations, we first recall the following lemma in \cite[Proposition 2.2]{ds1}.
\begin{lemma}\label{lem2.1}
There exists a constant $r>0$ such that for any positive closed current $T$ of bidegree $(1,1)$ with mass $1$ on $(X, \omega)$,
there is a smooth $(1,1)$-form $\alpha$ which depends only on the cohomology class of $T$ and a q.p.s.h. function $\varphi$
satisfying that
\begin{equation*}
-r\omega\leq\alpha\leq r\omega, \quad dd^{c}\varphi-T=\alpha.
\end{equation*}
\end{lemma}
Denote by $r(X, \omega)$ the smallest $r$ in Lemma \ref{lem2.1}.
For example, $r(\mathbb{CP}^{N},\omega_{FS})=1$.
Consider a positive measure $\mu$ on $X$. $\mu$ is said to be a PLB measure if all q.p.s.h. functions are integrable with respect to $\mu$.
It is easy to see that all moderate measures are PLB.
Now given a PLB probability measure $\mu$ on $X$ and $t\in\mathbb{R}$, we define,
\begin{equation}\label{e-7283-2}
\begin{split}
Q(X, \omega):&=\{\varphi ~q.p.s.h. ~on~ X, dd^{c}\varphi\geq -r(X,\omega)\omega\}, \\
R(X, \omega, \mu):&=\sup\{\max_{X}\varphi: \varphi\in Q(X, \omega), \int_{X}\varphi d\mu=0\} \\
&=\sup\{-\int_{X}\varphi d\mu: \varphi\in Q(X, \omega), \max_{X}\varphi=0\}, \\
S(X, \omega, \mu):&=\sup\{\bigl|\int\varphi d\mu \bigr|: \varphi\in Q(X, \omega), \int_{X}\varphi\omega^{n}=0\}, \\
\Delta(X, \omega, \mu, t):&=\sup\{\mu(\varphi<-t): \varphi\in Q(X, \omega), \int_{X}\varphi d\mu=0\}. \\
\end{split}
\end{equation}
These constants are related to Alexander-Dinh-Sibony capacity, see \cite[A.2]{ds1}.

Let $\Phi_{p}$ be a sequence of meromorphic transforms from a projective manifold $(X, \omega)$ into
the compact K\"{a}hler manifolds $(\mathbb{X}_{p}, \omega_{p})$ of the same codimension $k$, where
$\mathbb{X}_{p}$ is defined in \eqref{e-7283}.
Let
\begin{equation*}
d_{p}=d_{1,p}+...+d_{m,p}
\end{equation*}
be the dimension of $\mathbb{X}_{p}$.
Consider a PLB probability measure $\mu_{p}$ on $\mathbb{X}_{p}$,
for every $p>0, \epsilon>0$, we define
\begin{equation}\label{e-7283-3}
E_{p}(\epsilon):=\bigcup_{\|\phi\|_{\mathscr{C}^{2}}\leq 1}\{s_{p}\in\mathbb{X}_{p}:
\bigl|\bigl<\Phi_{p}^{\ast}(\delta_{s_{p}})-\Phi_{p}^{\ast}(\mu_{p}), \phi\bigr>\bigr|\geq \delta^1(\Phi_{p})\epsilon\},
\end{equation}
where $\delta_{s_{p}}$ is the Dirac measure at the point $s_{p}$.
By the definition of the pullback of $\Phi_{p}$ on currents, we see that $\Phi_{p}^{\ast}(\delta_{s_{p}})$
and $\Phi_{p}^{\ast}(\mu_{p})$ are positive closed currents of bidimension $(k,k)$ on $X$.
Moreover, $\Phi_{p}^{\ast}(\delta_{s_{p}})$ is well-defined for $s_{p}\in\mathbb{X}_{p}$ generic.

Recall that $\omega_{Mp}$ and $c_p$ were defined in \cite[(18),(19)]{sh2}.
The following estimate from Dinh-Sibony equidistribution theorem \cite{ds1} is crucial in our paper.
\begin{theorem}\label{thm2.2}
Let $\eta_{\epsilon, p}:=\epsilon\delta^2(\Phi_{p})^{-1}\delta^1(\Phi_{p})-3R(\mathbb{X}_{p}, \omega_{Mp}, \mu_{p})$,
then
\begin{equation*}
\mu_{p}(E_{p}(\epsilon))\leq\Delta(\mathbb{X}_{p}, \omega_{Mp}, \mu_{p}, \eta_{\epsilon, p}).
\end{equation*}
\end{theorem}
 We also need the following important estimate, which was deduced from \cite[Lemma 4.2(c),Proposition 4.3]{ds1}.
\begin{theorem}\label{thm2.3}
In the above setting, we have
\begin{equation*}
\bigl|\bigl<\delta^1(\Phi_{p})^{-1}(\Phi_{p}^{\ast}(\mu_{p})
-\Phi_{p}^{\ast}(\omega_{Mp}^{d_{p}})), \phi\bigr>\bigr|\leq
2S(\mathbb{X}_{p}, \omega_{Mp}, \mu_{p})\delta^2(\Phi_{p})\delta^1(\Phi_{p})^{-1}\|\phi\|_{\mathscr{C}^{2}}
\end{equation*}
for any $(k,k)$-form $\phi$ of class $\mathscr{C}^{2}$ on $X$.
\end{theorem}

\begin{theorem}\label{thm2.4}
Suppose that the sequence $\{R(\mathbb{X}_{p}, \omega_{Mp}, \mu_{p})\delta^2(\Phi_{p})\delta^1(\Phi_{p})^{-1}\}$ tends to $0$ and
\begin{equation*}
\Sigma_{p\geq 1}\Delta(\mathbb{X}_{p}, \omega_{Mp}, \mu_{p}, \delta^2(\Phi_{p})^{-1}\delta^1(\Phi_{p})t)<\infty
\end{equation*}
for all $t>0$.
Then for almost everywhere $s=(s_{p})\in\mathbb{P}^{X}$ with respect to $\mu=\prod\limits_{p=1}^\infty\mu_p$, the sequence $\langle \delta^1(\Phi_{p})^{-1}(\Phi_{p}^{\ast}(\delta_{s_{p}})-\Phi_{p}^{\ast}(\mu_{p})), \phi \rangle$
converges to $0$ uniformly on the bounded set of $(k-1,k-1)$-forms on $X$ of class $\mathscr{C}^{2}$.
\end{theorem}

Consider the $Kodaira$ $map$
\begin{equation*}
\Phi_{k,p}: X\rightarrow\mathbb{CP}(H_{(2)}^{0}(X, L_{kp})^{\ast}).
\end{equation*}
Here $H_{(2)}^{0}(X, L_{kp})^{\ast}$ is the dual space of $H_{(2)}^{0}(X, L_{kp}))$.
Choose $\{S_{k,p}^{j}\}_{j=0}^{d_{k,p}}$ as an orthonormal basis of $H_{(2)}^{0}(X, L_{kp}))$.
By an identification via the basis, it boils down to a meromorphic map
\begin{equation*}
\Phi_{k,p}: X\rightarrow \mathbb{CP}^{d_{kp}}.
\end{equation*}
Now we give a local analytic description of the above map.
Let $U\subset X$ be a contractible Stein open subset, $e_{kp}$ be a local holomorphic frame of $L_{kp}$ on $U$.
Then there exists a holomorphic function $s_{j}^{k,p}$ on $U$ such that $S_{k,p}^{j}=s_{j}^{k,p}e_{kp}$.
Then the map is expressed locally as
\begin{equation}\label{e-7283-4}
\Phi_{k,p}(x)=[s_{0}^{k,p}(x):...:s_{d_{k,p}}^{k,p}(x)], \quad \forall x\in U
\end{equation}
It is called the Kodaira map defined by the basis $\{S_{k,p}^{j}\}_{j=0}^{d_{k,p}}$.
Denote by $B_{kp}$ the Bergman kernel function defined by
\begin{equation}
B_{kp}(x)=\sum_{j=0}^{d_{k,p}}|S_{k,p}^{j}(x)|^{2}_{h_{k,p}}, \quad |S_{k,p}^{j}(x)|^{2}_{h_{k,p}}=h_{k,p}(S_{k,p}^{j}(x),S_{k,p}^{j}(x)).
\end{equation}
It is easy to see that this definition is independent of the choice of basis.

Recall that $\omega_{FS}$ is the normalized Fubini-Study form on $\mathbb{CP}^{d_{k,p}}$. We define the Fubini-Study
currents $\gamma_{k,p}$ of $H_{(2)}^{0}(X, L_{kp})$ as pullbacks of $\omega_{FS}$ by Kodaira map,
\begin{equation}\label{e-7283-5}
\gamma_{k,p}=\Phi_{k,p}^{\ast}(\omega_{FS}).
\end{equation}
We have in the local Stein open subset $U$,
\begin{equation*}
\gamma_{k,p}\bigl|_{U}=\frac{1}{2}dd^{c}\log\sum_{j=0}^{d_{kp}}|s_{j}^{k,p}|^{2}.
\end{equation*}
This yields
\begin{equation*}
\frac{1}{p}\gamma_{k,p}=c_{1}(L_{k},h_{k})+\frac{1}{2p}dd^{c}\log B_{kp}.
\end{equation*}
Since $\log B_{kp}$ is a global function which belongs to $L^{1}(X, \omega^{n})$,
$\frac{1}{p}\gamma_{k,p}$ has the same cohomology class as $c_{1}(L_{kp},h_{kp})$. We focus on the special meromorphic transforms $\Phi_{p}:X\rightarrow \mathbb{X}_p$ induced by the product map of Kodaira maps $ \Phi_{kp}: X\rightarrow \mathbb{CP}H_{(2)}^0(X,L_{kp})$. $\Phi_p$ is indeed a meromorphic transform with a graph
\begin{equation*}
\Gamma_{kp}=\{(x, s_{1p},\cdots,s_{mp})\in X\times \mathbb{X}_p: s_{1p}(x)=\cdots=s_{mp}(x)=0\},
\end{equation*}
see \cite[Section 3]{sh2}.
We also need the following
\begin{lemma}\label{lem2.5}
$\Phi_{p}^{\ast}(\delta_{s_{p}})=[s_{p}=0]$.
\end{lemma}
\begin{proposition}\label{pro2.6} \emph{\cite[Lemma 4.5]{cmn}}
$\Phi_{p}^{\ast}(\omega_{Mp}^{d_{p}})=\gamma_{1,p}\wedge...\wedge\gamma_{m,p}$ for all $p$ sufficiently large.
\end{proposition}

\subsection{Mderate measures on multi-projective spaces}
We say that a function $\phi$ on $X$ is quasi-plurisubharmonic (q.p.s.h) if it is $c\omega$-p.s.h. for some constant $c>0$.
Consider a measure $\mu$ on $X$, $\mu$ is said to be PLB if all the q.p.s.h. functions are $\mu$-integrable.

Let
\begin{equation}\label{e-7283-6}
\mathcal{F}=\{\phi ~q.p.s.h. ~on~ X: dd^{c}\phi\geq -\omega, \max_{X}\phi =0\}.
\end{equation}
$\mathcal{F}$ is compact in $L^{p}(X)$ and bounded in $L^{1}(\mu)$ when $\mu$ is a PLB measure, see \cite{ds1}.
\begin{definition}\label{def2.7}
Let $\mu$ be a {\rm PLB} measure on $X$. We say that $\mu$ is $(c,\alpha)$-moderate for some constants $c, \alpha >0$ if
\begin{equation*}
\int_{X}\exp(-\alpha \phi)d\mu\leq c
\end{equation*}
for all $\phi\in\mathcal{F}$. The measure $\mu$ is called moderate if there exist constants $c, \alpha >0$ such that it is $(c,\alpha)$-moderate.
\end{definition}
For example, $\omega^n$ is moderate in $X$. In particular, the Fubini-Study volume form is moderate in a projective space. We introduce product of moderate measures used in the main theorem. We define singular moderate measures $\sigma_{p}$ as perturbations of standard measures on $\mathbb{X}_{p}$.
For each $p\geq 1, 1\leq k\leq m, 1\leq j\leq d_{k,p}$, let $u_{j}^{kp}:\mathbb{CP}H_{(2)}^{0}(X, L_{kp})\rightarrow\mathbb{R}$
be an upper-semi continuous function. Fix $0<\rho<1$ and a sequence of positive constants $\{c_{p}\}_{p\geq 1}$.
We call $\{u_{j}^{k,p}\}$ {\it a family of $(c_{p},\rho)$-functions} if all $u_{j}^{k,p}$
satisfy the following two conditions:
\begin{flushleft}
$\bullet$ $u_{j}^{k,p}$ is of class $\mathscr{C}^{\rho}$ with modulus $c_{p}$, \\
$\bullet$ $u_{j}^{k,p}$ is a $c_{p}\omega_{FS}$-p.s.h.
\end{flushleft}
Then for each $p\geq 1$, there is a probability measure
\begin{equation}\label{e-7281}
\sigma_{p}=\prod_{k=1}^{m}\bigwedge_{j=1}^{d_{kp}}\pi_{k,p}^{\ast}(dd^{c}u_{j}^{k,p}+\omega_{FS})
\end{equation}
on $\mathbb{X}_{p}$.  By \cite[Theorem 1.1, Remark 2.12]{sh1},
$\bigwedge_{j=1}^{d_{k,p}}(dd^{c}u_{j}^{k,p}+\omega_{FS})$ is a moderate measure on $\mathbb{CP}H_{(2)}^{0}(X, L_{kp})$
when $c_{p}\leq 1/c^{(\sum_{k=1}^m A_{kp})^{n}}$ for a suitable constant $c>1$, $\forall 1\leq k\leq m, p\geq 1$.
We call
\begin{equation}\label{e-7282}
\sigma=\prod_{p=1}^{\infty}\sigma_{p}=\prod_{p=1}^{\infty}\prod_{k=1}^{m}\bigwedge_{j=1}^{d_{kp}}\pi_{k,p}^{\ast}(dd^{c}u_{j}^{k,p}+\omega_{FS})
\end{equation}
a probability measure on $\mathbb{P}^{X}$ generated by a family of $(c_{p},\rho)$-functions $\{u_{j}^{k,p}\}$ on $\{\mathbb{CP}H_{(2)}^{0}(X, L_{kp})\}$.

\subsection{Capacity estimate on multi-projective space}
Now we study the estimates on multi-projective spaces.
Let $\mathbb{CP}^{\ell_{1}},...,\mathbb{CP}^{\ell_{m}}$ be $m$ projective spaces.
Let $\pi_{k}: \mathbb{CP}^{\ell_{1}}\times...\times\mathbb{CP}^{\ell_{m}}\rightarrow \mathbb{CP}^{\ell_{k}}$
be the natural projection map.
Let $\sigma_{k}$ be a probability moderate measure with respect to a family of $(c_{\ell_{k}},\rho)$-functions
$\{u_{k,j}\}_{j=1}^{\ell_{k}}$ on $\mathbb{CP}^{\ell_{k}}$ defined in \eqref{e-7281}.
Let $\ell=\ell_{1}+...+\ell_{m}$.

Recall that the notation $r(\mathbb{CP}^{\ell_{1}}\times...\times \mathbb{P}^{\ell_{m}}, \omega_{Mp})$
is defined after Lemma \ref{lem2.1}.
We have the following lemma \cite[Lemma 4.6]{cmn}.
\begin{lemma}\label{lem2.8}
Under the above hypotheses,
\begin{equation*}
r(\mathbb{CP}^{\ell_{1}}\times...\times \mathbb{CP}^{\ell_{m}}, \omega_{Mp})\leq r(\ell_{1},...\ell_{m}):=\max_{1\leq k\leq m}\frac{\ell}{\ell_{k}}.
\end{equation*}
\end{lemma}
The following proposition is taken from \cite[Proposition 3.14]{sh2}.
\begin{proposition}\label{pro2.9}
In the above setting, let $\mathbb{CP}^{\ell_{k}}$ be a projective space endowed
with a probability moderate measure $\sigma_{k}$ defined in \eqref{e-7281}, $\forall 1\leq k\leq m$.
Set $\sigma:=\sigma_{1}\times...\times\sigma_{m}$.
Suppose that $\ell_{1},..,\ell_{m}$ are chosen sufficiently large such that
\begin{equation}\label{e-7282-1}
\begin{split}
\frac{r(\ell_{1},...,\ell_{m})\log\ell}{\min(\ell_{1},...,\ell_{m})}&\ll 1, \\
(\frac{\rho}{4})^{\min(\ell_{1},...,\ell_{m})}\ell&\ll 1. \\
\end{split}
\end{equation}
Then there exist positive constants $\beta_{1}, \beta_2, \xi$ depending only on $m$ such that for $0\leq t\leq \min(\ell_{1},...,\ell_{m})$,
we have
\begin{equation*}
\begin{split}
R(\mathbb{CP}^{\ell_{1}}\times...\times \mathbb{CP}^{\ell_{m}}, \omega_{Mp}, \sigma)&\leq \beta_{1}r(\ell_{1},...\ell_{m})(1+\log\ell), \\
S(\mathbb{CP}^{\ell_{1}}\times...\times \mathbb{P}^{\ell_{m}}, \omega_{Mp}, \sigma)&\leq \beta_{1}r(\ell_{1},...\ell_{m})(1+\log\ell), \\
\Delta(\mathbb{CP}^{\ell_{1}}\times...\times \mathbb{CP}^{\ell_{m}}, \omega_{Mp}, \sigma, t)&\leq \beta_{1}\ell^{\xi}\exp(-\beta_2 t/r(\ell_{1},...\ell_{m})). \\
\end{split}
\end{equation*}
\end{proposition}

\section{Convergence of Fubini-study currents}

Recall that the Fubini-Study current of $H_{(2)}^0(X,L_{kp})$ is
\begin{equation*}
\gamma_{kp}=\frac{1}{2} dd^c\log \sum_{j=0}^{d_{kp}}\|f_{kp,j}\|^2
=c_1(L_{kp},h_{kp})+\frac{1}{2}dd^c\log B_{kp}.
\end{equation*}

In this section, we will prove the following.
\begin{proposition}\label{pro3.1}
With the same notations and assumptions in Theorem 1, there exists $C>0$ such that
\begin{equation*}
\begin{split}
&|<\frac{1}{\prod_{k=1}^m A_{kp}}\gamma_{1p}\wedge\cdots\wedge \gamma_{mp}-\omega_{1}\wedge\cdots\wedge \omega_{n}, \phi>|\\
&\leq  C \sum_{k=1}^m\bigl(\frac{\log A_{kp}}{A_{kp}}+A_{kp}^{-a_k}\bigr)\|\phi\|_{\mathcal{C}^2}.
\end{split}
\end{equation*}
for any $(n-m,n-m)$-form $\phi$ of class $\mathcal{C}^2$.
\end{proposition}
\begin{proof}
Let $W_p:= \frac{1}{\prod_{k=1}^m A_{kp}} \gamma_{1p}\wedge\cdots\wedge \gamma_{mp}-\omega_1\wedge\cdots\wedge\omega_m$, $\alpha_{kp}=\frac{c_1(L_{kp},h_{kp})}{A_{kp}}-\omega_k$.
By Assumption 2, $\|\alpha_{kp}\|\leq \frac{C_0}{A_{kp}^{a_k}}$ in the norm of currents
\begin{equation*}
\frac{\gamma_{kp}}{A_{kp}}-\omega_k=\alpha_{kp}+\frac{1}{2A_{kp}}dd^c \log B_{kp}.
\end{equation*}
We have
\begin{equation*}
\begin{split}
&|<W_p,\phi>|\\
&=|\sum\limits_{k=1}^m <\omega_1\wedge\cdots\wedge \omega_{k-1}\wedge (\frac{\gamma_{kp}}{A_{kp}}-\omega_{k})\wedge\frac{\gamma_{k+1,p}}{A_{k+1,p}}\wedge\cdots\wedge\frac{\gamma_{m,p}}{A_{mp}},\phi>|\\
&\leq \sum\limits_{k=1}^m |<\omega_1\wedge\cdots\wedge \omega_{k-1}\wedge (\alpha_{kp}+\frac{1}{2A_{kp}}dd^c \log B_{kp})\wedge \frac{\gamma_{k+1,p}}{A_{k+1,p}}\wedge\cdots\wedge\frac{\gamma_{m,p}}{A_{mp}},\phi>|.
\end{split}
\end{equation*}
Note that there exists a constant $c>0$ such that
\begin{equation*}
-\frac{cC_0}{A_{kp}}\omega^{n-m+1}\leq \alpha_{kp}\wedge \phi \leq \frac{cC_0}{A_{kp}^{a_k}}\omega^{n-m+1}
\end{equation*}
in the sense of currents. Then,
\begin{equation*}
\begin{split}
&|<W_p,\phi>|\\
&= |\sum\limits_{k=1}^m \frac{cC_0}{A_{kp}^{a_k}}\|\phi\|_{\mathcal{C}^0}\int_{X}|\omega_1\wedge\cdots\wedge\omega_{k-1}\wedge\omega^{n-m+1}\wedge\frac{\gamma_{k+1,p}}{A_{k+1,p}}\wedge\cdots\wedge\frac{\gamma_{m,p}}{A_{m,p}}|\\ &+\sum\limits_{k=1}^m \int_{X}|\frac{\log B_{kp}}{2A_{kp}}dd^c\phi\wedge (\omega_1\wedge\cdots\wedge\omega_{k-1}\wedge\frac{\gamma_{k+1,p}}{A_{k+1,p}}\wedge\cdots\wedge\frac{\gamma_{m,p}}{A_{m,p}})|\\
&=I+II.
\end{split}
\end{equation*}

We choose $A_p\geq M_0$ for large $p$,
\begin{equation*}
A_{kp}^{n-1}\leq B_{kp}(x)\leq A_{kp}^{n+1},
\end{equation*}
so $|\log B_{kp}|\leq (n+1)\log A_{kp}$.
Hence
\begin{equation*}
\begin{split}
II= |\sum\limits_{k=1}^m \frac{nc\log A_{kp}}{A_{kp}}\|\phi\|_{\mathcal{C}^2}\int_{X}|\omega_1\wedge\cdots\wedge\omega_{k-1}\wedge\omega^{n-m+1}\wedge\frac{\gamma_{k+1,p}}{A_{k+1,p}}\wedge\cdots\wedge\frac{\gamma_{m,p}}{A_{m,p}}|.
\end{split}
\end{equation*}
Then we have
\begin{equation*}
\begin{split}
&|<W_p,\phi>|\\
= &|\sum\limits_{k=1}^m (\frac{nc\log A_{kp}}{A_{kp}}+\frac{cC_0}{A_{kp}^{a_k}})\|\phi\|_{\mathcal{C}^2}| \int_{X}\omega_1\wedge\cdots\wedge\omega_{k-1}
\wedge\omega^{n-m+1}\wedge\frac{\gamma_{k+1,p}}{A_{k+1,p}}\wedge\cdots\wedge\frac{\gamma_{m,p}}{A_{m,p}}|\\
=&|\sum\limits_{k=1}^m (\frac{nc\log A_{kp}}{A_{kp}}+\frac{cC_0}{A_{kp}^{a_k}})\|\phi\|_{\mathcal{C}^2}\\
&\int_{X}|\omega_1\wedge\cdots\wedge\omega_{k-1}
\wedge\omega^{n-m+1}\wedge\frac{c_1(L_{k+1,p},h_{k+1,p})}{A_{k+1,p}}\wedge\cdots\wedge\frac{c_1(L_{m,p},h_{m,p})}{A_{m,p}}|\\
\leq & \sum\limits_{k=1}^m 2^{m-k} (\frac{nc\log A_{kp}}{A_{kp}}+\frac{cC_0}{A_{kp}^{a_k}})\|\phi\|_{\mathcal{C}^2}\int_{X}|\omega_1\wedge\cdots\wedge\omega_{k-1}
\wedge\omega^{n-m+1}\wedge\omega_{k+1}\wedge\cdots\wedge\omega_{m}|.
\end{split}
\end{equation*}
The proof is completed.
\end{proof}

\section{Proof of main theorems}
Our proof is based on Dinh-Sibony's technique on equidistribution theorems. The key estimate is the following.
\begin{equation}\label{e-7291}
\begin{split}
&|\langle\frac{1}{\prod_{k=1}^m A_{kp}}[s_p=0]-\omega_1\wedge\cdots\wedge\omega_m,\phi\rangle|\\
&\leq  |\langle\frac{1}{\prod_{k=1}^m A_{kp}}(\Phi_p^{\ast}(\delta_{s_p})-\Phi_p^{\ast}(\sigma_p),\phi\rangle|\\
&+|\langle\frac{1}{\prod_{k=1}^m A_{kp}}(\Phi_p^{\ast}(\sigma_p)-\Phi_p^{\ast}(\omega_{Mp}^{d_p}),\phi\rangle|\\
&+ |\langle\frac{1}{\prod_{k=1}^m A_{kp}}(\Phi_p^{\ast}(\omega_{Mp}^{d_p})-\omega_1\wedge\cdots\wedge\omega_m,\phi\rangle|\\
&= I+ II+III.
\end{split}
\end{equation}
The estimate of $III$ is already done in Section 3.
Now we deal with $II$. First, we compute $\delta_p^1$ and $\delta_p^2$,
where
\begin{equation}\label{e-7292}
\begin{split}
\delta_p^1&=\int_X (\Phi_p^{\ast}(\omega_{Mp}^{d_p})\wedge\omega^{n-m}, \\
\delta_p^2&=\int_X (\Phi_p^{\ast}(\omega_{Mp}^{d_p-1})\wedge\omega^{n-m+1}.
\end{split}
\end{equation}
\begin{lemma}\label{lem4.1}
There exists a constant $C>1$, such that
\begin{equation*}
 \frac{\prod\limits_{k=1}^m A_{kp}}{C}\leq \delta_p^1\leq C \prod_{k=1}^m A_{kp},
\end{equation*}
\begin{equation*}
 \frac{1}{C}\frac{\prod\limits_{k=1}^m A_{kp}}{\sum\limits_{k=1}^m A_{kp}}\leq \delta_p^2\leq C\frac{\prod\limits_{k=1}^m A_{kp}}{\sum\limits_{k=1}^m A_{kp}}.
\end{equation*}
\end{lemma}
\begin{proof}
By Proposition \ref{pro2.6}, we have
\begin{equation*}
\begin{split}
\delta_p^1&=\int_X c_1(L_{1p},h_{1p})\wedge\cdots\wedge c_1(L_{mp},h_{mp})\wedge\omega^{n-m}\\
&\approx \prod\limits_{k=1}^m A_{kp}\int_X\omega_1\wedge\cdots\wedge_m\wedge\omega^{n-m}.
\end{split}
\end{equation*}
By Assumption 1, $\int_X\omega_1\wedge\cdots\wedge_m\wedge\omega^{n-m}>0$, then we have $\delta_p^1\approx \prod\limits_{k=1}^m A_{kp}$.

To compute $\delta_p^2$, we recall that
\begin{equation*}
\dim H^{\ell,\ell}(\mathbb{CP}^{N})=1,
\end{equation*}
\begin{equation*}
\dim H^{\ell_1,\ell_2}(\mathbb{CP}^{N})=0, \ell_1\neq \ell_2
\end{equation*}
for cohomology groups associated to sheaf of currents.
Then $\omega_{FS}^{d_{kp}}$ and $\delta_{s_{kp}}, \omega_{FS}^{d_{kp}-1}$ and $[\mathcal{D}_{kp}]$ have the same cohomology classes, where $\delta_{s_{kp}}$ and $\mathcal{D}_{kp}$ are generic point and complex line respectively in $\mathbb{CP}^{d_{kp}}$. By the definition of meromorphic transform, we have
\begin{equation*}
\begin{split}
\langle \Phi_{kp}^{\ast}([\mathcal{D}_{kp}]), \phi\rangle
= & \langle (\pi_1)_{\ast} (\pi_2)^{\ast}[\mathcal{D}_{kp}]\wedge[\Gamma_{kp}], \phi\rangle\\
=&(\pi_2)^{\ast}[\mathcal{D}_{kp}]\wedge[\Gamma_{kp}], (\pi_1)^{\ast} \phi\rangle\\
=& \langle [\pi_2^{-1}([\mathcal{D}_{kp})\bigcap \Gamma_{kp}], \pi_1^{\ast}\phi\rangle\\
= &\int_{\pi_2^{-1}(\mathcal{D}_{kp})\bigcap \Gamma_{kp}} \pi_1^{\ast}\phi.
\end{split}
\end{equation*}
Note that $\pi_2^{-1}(\mathcal{D}_{kp})\bigcap \Gamma_{kp}=\{(x, s_{kp})\in X\times \mathcal{D}_{kp}, s_{kp}(x)=0\}$.

Since $\mathcal{D}_{kp}$ is generic, $\forall x\in X$, there exists a unique $s_{kp}\in \mathcal{D}_{kp}$ such that $s_{kp}(x)=0$, where $\Gamma_{kp}=\{(x,s_{kp})\in X\times\mathbb{CP}^{d_{kp}}: s_{kp}(x)=0\}$.
So $\pi_1: \pi_2^{-1}(\mathcal{D}_{kp})\bigcap \Gamma_{kp}\rightarrow X$ is bijective. Hence, $\langle\Phi_{kp}^{\ast}([\mathcal{D}_{kp}]), \phi\rangle=\int_X \phi$, i.e. $\Phi_{kp}^{\ast}([\mathcal{D}_{kp}])=[X]=1$.

Note that
\begin{equation*}
\begin{split}
\Phi_{kp}^{*}(\omega_{Mp}^{d_p-1})
= & \sum\limits_{k=1}^m\frac{d_{kp}}{c_pd_p}\Phi_{p}^{\ast}(\{s_{1p}\}\times\cdots\times\{s_{k-1,p}\}\times\{\mathcal{D}_{kp}\}\times\cdots\times\{s_{mp}\})\\
=&\sum\limits_{k=1}^m\frac{d_{kp}}{c_pd_p}[s_{1p}=0]\wedge\cdots\wedge\Phi_{kp}^{\ast}([\mathcal{D}_{kp}])\wedge\cdots\wedge[s_{mp}=0],
\end{split}
\end{equation*}
where the bounded sequence $\{c_p\}$ is mentioned before Theorem \ref{thm2.2}.
Then,
\begin{equation*}
\begin{split}
\delta_p^2&=\sum\limits_{k=1}^m\frac{d_{kp}}{c_pd_p}\int_{X} c_1(L_{1p},h_{1p})\wedge\cdots\wedge\widehat{c_1(L_{kp},h_{kp})}\wedge\cdots\wedge c_1(L_{mp},h_{mp})\wedge\omega^{n-m+1}\\
&\approx \sum\limits_{k=1}^m\frac{d_{kp}}{c_pd_p}\frac{\prod\limits_{j=1}^mA_{jp}}{A_{kp}}\int_X \omega_1\wedge\cdots\wedge\widehat{\omega_k}\wedge\omega_m\wedge^{n-m+1}\\
&\approx \frac{\prod\limits_{k=1}^mA_{kp}}{\sum\limits_{k=1}^m A_{kp}}.
\end{split}
\end{equation*}
The last approximation follows from the fact that $\{A_{kp}\}_{p=1}^{\infty}$ have the same infinity order.
The proof is completed.
\end{proof}

Now we are in a position to estimate the term $II$.
\begin{proposition}\label{pro4.2}
 We have
\begin{equation*}
|\langle\frac{1}{\prod_{k=1}^mA_{kp}}(\Phi_p^{\ast}(\sigma_p)-\Phi_p^{\ast}(\omega_{Mp}^{d_p})), \phi\rangle|\leq \frac{C\log(\sum\limits_{k=1}^mA_{kp})}{\sum\limits_{k=1}^mA_{kp}}\|\phi\|_{\mathcal{C}^2}.
\end{equation*}
\end{proposition}
\begin{proof}
Recall that $d_{kp}\approx A_{kp}^n$ and $\{d_{1p}\},\cdots,\{d_{mp}\}$ satisfy the conditions in Proposition \ref{pro2.9} due to Assumption 1. By Lemma \ref{lem4.1} and Proposition \ref{pro2.9} we deduce that $S(\mathbb{X}_p, \omega_{Mp},\sigma_p)\leq C\log(\sum\limits_{k=1}^mA_{kp})$.
By Lemma \ref{lem4.1},
\begin{equation*}
\delta_p^2/\delta_p^1\approx \frac{1}{\sum\limits_{k=1}^m A_{kp}}.
\end{equation*}
Then it follows from Theorem \ref{thm2.3} that
\begin{equation*}
\begin{split}
&|\frac{1}{\prod\limits_{k=1}^m A_{kp}}(\Phi_{p}^{\ast}(\sigma_p)-\Phi_{p}^{\ast}(\omega_{Mp}^{d_p}),\phi\rangle|\\
\leq &CS(\mathbb{X}_p,\omega_{Mp},\sigma_p)(\delta_p^2/\delta_p^1)\|\phi\|_{\mathcal{C}^2}\\
\leq &C\frac{\log(\sum\limits_{k=1}^mA_{kp})}{\sum\limits_{k=1}^mA_{kp}}\|\phi\|_{\mathcal{C}^2}.
\end{split}
\end{equation*}
The proof is completed.
\end{proof}
Next we study the estimate of the term $I$.
\begin{proposition}\label{pro4.3}
For $\sigma$-a.e.$\{s_p\}\in \mathbb{P}^{X}$, we have
\begin{equation*}
\frac{1}{\prod\limits_{k=1}^m A_{kp}}(\Phi_{p}^{\ast}(\delta_{s_p})-\Phi_{p}^{\ast}(\sigma_p))
\end{equation*}
tens to $0$.
\end{proposition}

\begin{proof}
By Lemma \ref{lem2.8} and Proposition \ref{pro2.9}, we have
\begin{equation*}
R(\mathbb{X}_p,\omega_{Mp},\sigma_p)\leq C(1+\log(\sum\limits_{k=1}^m A_{kp})),
\end{equation*}
\begin{equation*}
\Delta(\mathbb{X}_p,\omega_{Mp},\sigma_p,t)\leq C(\sum\limits_{k=1}^m A_{kp})^{\xi}\exp(-\tilde\beta_2t),
\end{equation*}
where $\tilde\beta_2$ is a positive constant.
Then
\begin{equation*}
R(\mathbb{X}_p,\omega_{Mp},\sigma_p)\delta_p^2/\delta_p^1 \rightarrow 0,
\end{equation*}
\begin{equation*}
\sum\limits_{p=1}^{\infty}\Delta(\mathbb{X}_p,\omega_{Mp},\sigma_p,(\delta_p^2/\delta_p^1)t)\leq C\sum\limits_{p=1}^{\infty}(\sum\limits_{k=1}^m A_{kp})^{\xi}\exp(-\tilde\beta_2(\sum\limits_{k=1}^m A_{kp})t)<\infty.
\end{equation*}
Then the proof is completed by applying Theorem \ref{thm2.4}.
\end{proof}

\begin{proof}[End of the proof of Theorem \ref{thm1.1}]: The theorem follwos from Proposition \ref{pro3.1}, Proposition \ref{pro4.2} and Proposition \ref{pro4.3}.
\end{proof}
Now we prove Theorem \ref{thm1.2} by applying Theorem \ref{thm2.2}, which gives also an alternative proof of Theorem \ref{thm1.1}.
\begin{proof}
We take $C_4>0$ to be determined later and set
\begin{equation*}
\varepsilon_p:=\frac{C_4\log(\sum\limits_{k=1}^mA_{kp})}{\sum\limits_{k=1}^mA_{kp}},
\end{equation*}
and
\begin{equation*}
\begin{split}
\eta_{\varepsilon p} =&
\varepsilon_p \delta_p^1/\delta_p^2-3R_p\\
\geq
&C_5\varepsilon_p\log(\sum\limits_{k=1}^mA_{kp})-C\log(\sum\limits_{k=1}^mA_{kp})\\
\geq &C_6\log(\sum\limits_{k=1}^mA_{kp}),
\end{split}
\end{equation*}
where $C_6>0$ determined by $C_4$.
Note that $\log(\sum\limits_{k=1}^mA_{kp})\leq
\min\{\sum\limits_{k=1}^mA_{kp}\}$ by Assumption 1 for large $p$.
We can apply Theorem \ref{thm2.2} and derive that
\begin{equation*}
\begin{split}
\sigma_p(E_p({\varepsilon_p})) \leq
&\Delta(X_p,\omega_{Mp},\sigma_p,\eta_{\varepsilon p})\\
\leq &C_1(\sum\limits_{k=1}^mA_{kp})^{\xi}(\sum\limits_{k=1}^mA_{kp})^{-\tilde\beta_2 C_6}\\
= &C_1(\sum\limits_{k=1}^mA_{kp})^{-\alpha},
\end{split}
\end{equation*}
where $C_4$ is chosen such that $\alpha=\tilde\beta_2 C_6-\xi>0$. Since $\tilde\beta_2$ and $\xi$ are fixed constants, $\alpha>0$ can be arbitrarily chosen.
Set $E_p^\alpha :=E_p(\varepsilon_p)$.
Hence for any $s_p\in X_p\setminus E_p^\alpha$, we have
\begin{equation}\label{e-7293}
\begin{split}
&|\frac{1}{\prod\limits_{k=1}^mA_{kp}}\langle[s_p=0]-\Phi_p^{\ast}(\sigma_p),\phi\rangle|\\
&\leq
\frac{C_7\log(\sum\limits_{k=1}^mA_{kp})}{\sum\limits_{k=1}^mA_{kp}}\|\phi\|_{\mathscr{C}^{2}}.
\end{split}
\end{equation}
Combining Proposition \ref{pro3.1}, Proposition \ref{pro4.2} and \eqref{e-7293}, we obtain
\begin{equation*}
\begin{split}
&\left|\langle\frac{1}{\prod\limits_{k=1}^mA_{kp}}[s_p=0]-\omega_1\wedge \omega_2\cdots\wedge w_m,\phi\rangle\right|\\
&\leq C_2\left(\sum\limits_{k=1}^m
\frac{\log A_{kp}}{A_{kp}}+\frac{\log(\sum\limits_{k=1}^mA_{kp})}{\sum\limits_{k=1}^mA_{kp}}+\sum\limits_{k=1}^mA_{kp}^{-a_{k}}\right)\|\phi\|_{\mathscr{C}^{2}}.
\end{split}
\end{equation*}
Then the proof of Theorem \ref{thm1.2} is completed.
When
$\sum\limits_{p=1}^{\infty}(\sum\limits_{k=1}^mA_{kp})^{-\alpha}<\infty$,
we can prove Theorem 1 by standard arguments using Borel-Cantelli lemma (cf. \cite[Proposition 4.5]{sh2}).
\end{proof}

\section{Dimension growth of a sequence of pseudo-effective line bundles}.

In this section, we provide a dimension growth result which sheds a
light on Assumption 1.
It is enough to consider one sequence of holomorphic line bundles
$(L_p,h_p)$.
We are devoted to proving Theorem \ref{thm1.4}.

We first recall the $L^2$-estimate for line bundles with singular metrics (cf.\cite[Theorem 3.2]{cmn18}).

\begin{theorem}\label{thm5.1}
Let $(X,\omega)$ be a K\"{a}hler manifold of dimension $n$ which admits a complete K\"{a}hler metric. Let $(L,h)$ be a singular Hermitian holomorphic line bundle and let $\lambda: X\rightarrow [0,\infty)$ be a continuous function such that $c_1(L,h)\geq \lambda\omega$. Then for any form $g\in L_{0,1}^2(X,L,loc)$ satisfying
\begin{equation*}
\overline{\partial} g=0, \int_{X}\lambda^{-1}|g|^2_{h}\omega^n <\infty,
\end{equation*}
there exists $u\in L^2(M,L,loc)$ with $\overline\partial u=g$ and
\begin{equation*}
\int_{X}|u|^2_{h}\omega^n \leq\int_{X}\lambda^{-1}|g|^2_{h}\omega^n.
\end{equation*}
\end{theorem}
\begin{proof}[Proof of Theorem \ref{thm1.4}:]
We add additional local weights to the sequence.
Let $x\in U_{\alpha}\Subset X\setminus\Sigma$, $e_p$ is the local
frame of $L_{p}$ on $U_{\alpha}$.
Fix $r_0>0$ so that the ball $V:=B(x,2r_0)\subset\subset U_{\alpha}$ and let $U:=B(x,r_0)$. Let $\theta\in \mathscr{C}^{\infty}(\mathbb{R})$ be a cut-off function such that $0\leq\theta\leq 1, \theta(t)=1$ for $|t|\leq \frac{1}{2}, \theta(t)=0$ for $|t|\geq 1$. For $z\in U$, define the quasi-psh function $\varphi_z$ on $X$ by
\begin{equation*}
\varphi_z(y))=
\left\{\begin{array}{lc}
\theta\left(\frac{|y-z|}{r_0})\log(\frac{|y-z|}{r_0}\right), \ \  \text{for} \ \ y\in U_{\alpha},\\
0, \ \  \text{for}\ \ y\in X\backslash B(z,r_0).\\
\end{array}\right.
\end{equation*}
Note that $dd^c \phi_z\geq 0$, on $\{y:|y-z|\leq \frac{r_0}{2}\}$. Since $V\Subset U_{\alpha}$, it follows that there exists a constant $c'>0$ such that for all $z\in U$ we have $dd^c\varphi_z\geq -c'\omega$ on $X$ and $dd^c\varphi_z=0$ outside $\overline{V}$.
Since
\begin{equation*}
c_1(L_p,h_p)\geq A_p\eta_pw\geq A_pc_x\omega=A_pc\omega.
\end{equation*}
We can find constants $a,b$ with $a=c-bc'>0$, such that
\begin{equation*}
\begin{split}
c_1(L_p,h_pe^{-bA_p\varphi_{z}})\geq & 0 \ \text{on} \  X.\\
c_1(L_p,h_pe^{-bA_p\varphi_{z}})= &
c_1(L_p,h_p)+bA_pdd^c\varphi_z\\
\geq & A_p(c-bc^{\prime})w=aA_pw \ \text{near} \ \overline{V}.
\end{split}
\end{equation*}

Consider a continuous function $\lambda_p: X\rightarrow
[0,+\infty)$ such that $\lambda=aA_p$ on $\overline{V}$,
$c_1(L_p,h_pe^{-bA_p\varphi_{z}})\geq\lambda_pw$.
Set $ \beta=(\beta_1,...,\beta_n)$ with
$\sum\limits_{j=1}^n\beta_j\leq [bA_p]-n$, and
\begin{equation*}
v_{z,p,\beta}(y)=(y_1-z_1)^{\beta_1}...(y_n-z_n)^{\beta_n}.
\end{equation*}
Let
\begin{equation*}
g_{z,p,\beta}=\overline{\partial}(v_{z,p,\beta}\theta(\frac{|y-z|}{r_0})e_p).
\end{equation*}
Then
\begin{equation*}
\begin{split}
&\int_X\frac{1}{\lambda}|g_{z,p,\beta}|^2_{h_p}e^{-2bA_p\varphi_{z}}\omega^n\\
&=\int_V\frac{1}{\lambda}|g_{z,p,\beta}|^2_{h_p}e^{-2bA_p\varphi_{z}}\omega^n\\
&= \frac{1}{aA_p}\int_{V\setminus
B(z,\frac{r_0}{2})}|v_{z,p,\beta}|^2|\partial\theta(\frac{|y-z|}{r_0})|^2
e^{-2\psi_p}e^{-2bA_p\varphi_{z}}\omega^n,
\end{split}
\end{equation*}
where $\psi_p$ is the local weight of $h_p$.

Note that $\varphi_z$ is bounded on $V\setminus B(z,\frac{r_0}{2})$. Then
\begin{equation*}
\int_X\frac{1}{\lambda}|g_{z,p,\beta}|^2_{h_p}e^{-2bA_p\varphi_{z}}\omega^n<\infty,\forall
p.
\end{equation*}
By applying Theorem \ref{thm5.1}, there exists
$u_{z,p,\beta}\in L^2(X,L_p)$,such that
\begin{equation*}
\overline{\partial}u_{z,p,\beta}=g_{z,p,\beta},
\end{equation*}
and
\begin{equation*}
\begin{split}
&\int_X|u_{z,p,\beta}|^2_{h_p}e^{-2bA_p\varphi_{z}}\omega^n\\
\leq&\int_X\frac{1}{\lambda}|g_{z,p,\beta}|^2_{h_p}e^{-2bA_p\varphi_{z}}\omega^n\\
\end{split}
\end{equation*}
So we construct an element
\begin{equation*}
S_{z,p,\beta}=v_{z,p,\beta}\theta(\frac{|y-z|}{r_0})
e_p-u_{z,p,\beta}
\end{equation*}
in $H_{(2)}^0(X,L_p)$.
For $y\in B(z,\frac{r_0}{2})$,
\begin{equation*}
S_{z,p,\beta}(y)=v_{z,p,\beta}(y) e_p-u_{z,p,\beta}(y),
\end{equation*}
we see that $u_{z,p,\beta}$ is a holomorphic near $z$.

Let $\mathscr{L}$ be the sheaf of holomorphic functions on $X$ vanishing at $z$ and let ${\bf m}\subset \mathscr{O}_{X,z}$ the maximal ideal of the ring of germs of holomorphic function at $z$.
Consider the natural map
\begin{equation*}
L_p\rightarrow L_p\otimes \mathscr{O}_X/\mathscr{L}^{a+1}.
\end{equation*}
This map induces a map in the level of cohomology
\begin{equation*}
J_p^a:H_{(2)}^0(X,L_p)\rightarrow H_{(2)}^0(X,L_p\otimes
\mathscr{O}_X/\mathscr{L}^{a+1})=(L_p)_z\otimes \mathscr{O}_{X,z}/\mathscr{L}^{a+1}.
\end{equation*}
The right hand side of the above map is called the space of $a$-jets of $L^2$-holomorphic sections of $L_p$ at $z$.
We recall the following fact:
\begin{equation*}
\int_{|y_1-z_1|<1,\cdots,|y_n-z_n|<1}\prod\limits_{k=1}^n|y_k-z_k|^{2r_k}|y-z|^{-2bA_p}i^ndy_1\wedge
d\overline{y}_1\wedge \cdots\wedge dy_n\wedge d\overline{y}_n<\infty,
\end{equation*}
if and only if $\sum\limits_{j=1}^nr_j\geq [bA_p]-n+1$.
Then for $u_{z,p,\beta}\in L^2(X,L_p)$, we have
\begin{equation*}
\int_X|u_{z,p,\beta}|^2e^{-2bA_p\varphi_{z}}\omega^n<\infty
\end{equation*}
if and only if $u_{z,p,\beta}$ has vanishing order of at least $[bA_p]-n+1$ at
$z$.
So the $([bA_p]-n)$-jet of $S_{z,p,\beta}$ coincides with
$v_{z,p,\beta}$.
For any such $v_{z,p,\beta},\sum\limits_{j=1}^n\beta_j\leq
[bA_p]-n$, we can construct $S_{z,p,\beta}$ as before such that
\begin{equation*}
J_p^{[bA_p]-n}(S_{z,p,\beta})=v_{z,p,\beta}.
\end{equation*}
So $J_p^{[bA_p]-n}$ is surjective.
Hence
\begin{equation*}
\begin{split}
d_p=&\dim H_{(2)}^0(X,L_p)-1\\
\geq & \dim (\mathscr{O}_{X,z}/\mathscr{L}^{[bA_p]-n+1})-1\\
= & \binom{[bA_p]}{[bA_p]-n}-1\geq C_3A_p^n,
\end{split}
\end{equation*}
for some constant $C_3>0$.
The proof is completed.
\end{proof}

\begin{remark}\label{rem5.1}
To get the upper estimate of $d_p$ by the spirit of Siegel's lemma, we need impose more conditions on the transition functions of each $L_p$.
\end{remark}

\noindent


\begin{thebibliography}{11}



\addcontentsline{toc}{chapter}{Biography}
\thispagestyle{plain}

%






\small
\bibitem{bbl} Bogomolny E., Bohigas O., Leboeuf P.,
 {\it Quantum chaotic dynamics and random polynomials},
  J. Stat. Phys., {\bf 85} (1996), no.5-6, 639--679.

\bibitem{bchm}  Bayraktar T., Coman D., Herrmann H., Marinescu G.,
 {\it A survey on zeros of random holomorphic sections},
    Dolomites Res. Notes Approx. 11 (2018), Special Issue Norm Levenberg, 1-19.

\bibitem{bcm}  Bayraktar T., Coman D., Marinescu G.,
 {\it Universality results for zeros of random holomorphic sections},
   Trans. Amer. Math. Soc. 373 (2020), no. 6, 3765-3791.



\bibitem{clmm} Coman D., Lu, W., Ma X., Marinescu G.,
 {\it Bergman kernels and equidistribution for sequences of line bundles on K\"{a}hler manifolds},
  preprint available at arXiv:2012.12019.

\bibitem{cm} Coman D., Marinescu G.,
 {\it Equidistribution results for singular metrics on line bundles},
  Ann. Sci. \'{E}cole Norm. Sup\'{e}r., {\bf 48} (2015), no.3, 497--536.

\bibitem{cmm} Coman D., Ma X., Marinescu G.,
 {\it Equidistribution for sequences of line bundles on normal K\"{a}hler spaces},
  Geometry and Topology, {\bf 21} (2017), 923--962.


\bibitem{cmn} Coman D., Marinescu G., Nguy\^{e}n V.-A.,
 {\it H\"{o}lder singular metrics on big line bundles and equidistribution},
  International Mathematics Research Notices, {\bf 16} (2016), 5048--5075.

\bibitem{cmn18} Coman D., Marinescu G., Nguy\^{e}n V.-A.,
 {\it Approximation and equidistribution results for pseudo-effective line bundles},
 J. Math. Pures Appl., {\bf 115}, (2018), 218--236.

\bibitem{cmn19} Coman D., Marinescu G., Nguy\^{e}n V.-A.,
 {\it Holomorphic sections of line bundles vanishing along subvarieties},
  preprint available at arXiv:1909.00328.

\bibitem{dmm} Dinh T.-C., Ma X., Marinescu G.,
 {\it Equidistribution and convergence speed for zeros of holomorphic sections of singular Hermitian line bundles},
  Journal of Functional Analysis, {\bf 271}, (2016), no.11, 3082--3110.

\bibitem{ds1}  Dinh T.-C., Sibony N.,
{\it Distribution des valeurs de transformations m\'{e}romorphes et applications},
  Comment. Math. Helv., {\bf 81} (2006), no. 5, 221--258.

\bibitem{hs} Hsiao C.-Y., Shao G.,
{\it Equidistribution theorems on strongly pseudoconvex domains},
  Trans. Amer. Math. Soc., {\bf 372}, (2019), no. 2, 1113--1137.


\bibitem{nv}  Nonnenmacher S., Voros A.,
{\it Chaotic eigenfunctions in phase space},
  J. Stat. Phys.,  {\bf 92}, (1998), no. 3, 451--518.


\bibitem{rs}  Rudnick Z., Sarnak P.,
{\it The behavior of eigenstates of arithmetic hyperbolic manifolds},
  Commun. Math. Phys.,  {\bf 161}, (1994), no. 1, 195--213.

\bibitem{sh1} Shao G.,
 {\it Equidistribution of zeros of random holomorphic sections for moderate measures},
 Math. Z., {\bf 283}, (2016), 791--806. 

\bibitem{sh2} Shao G.,
 {\it Equidistribution on big line bundles with singular metrics for moderate measures},
 J. Geom. Ana., {\bf 27}, (2017), 1295--1322. 

%
\bibitem{sz}  Shiffman B., Zelditch S.,
{\it Distribution of zeros of random and quantum chaotic sections of positive line bundles},
  Commun. Math. Phys.,  {\bf 200}, (1999), 661--683.
%


\end{thebibliography}
\end{document}